\documentclass{amsart}
\usepackage{amsthm,amssymb,amsmath,hyperref, esint}
\usepackage{graphicx}
\usepackage{color}
\usepackage{tikz}
\usetikzlibrary{decorations.pathreplacing}
\usepackage[utf8]{inputenc}
\usepackage{textgreek}
\usepackage{mathtools}

\definecolor{darkorchid}{rgb}{0.6, 0.2, 0.8}

\theoremstyle{plain}
\newtheorem{thm}{Theorem}[section]

\newtheorem{cor}[thm]{Corollary}
\newtheorem{lem}[thm]{Lemma}
\newtheorem{prop}[thm]{Proposition}

\theoremstyle{definition}
\newtheorem{defn}[thm]{Definition}

\theoremstyle{remark}
\newtheorem{rem}[thm]{Remark}

\theoremstyle{plain}

\numberwithin{equation}{section}

\newcommand{\N}{{\mathbb N}}

\newcommand{\Z}{{\mathbb Z}}

\def\udot#1{\ifmmode\oalign{$#1$\crcr\hidewidth.\hidewidth
    }\else\oalign{#1\crcr\hidewidth.\hidewidth}\fi}

\def\Z{\mathbb{Z}}

\def\l{\ell}

\newcommand\MOD[1]{\ \left(\normalfont\text{mod } {#1}\right)} 
\newcommand\cs[1]{\langle{#1}\rangle} 

\DeclareMathOperator*{\esssup}{ess \ sup}
\DeclareMathOperator*{\essinf}{ess \ inf}
\usepackage{float}

\title{The structure of weight and function classes with coprime bases}
\author{Theresa C. Anderson, Chiara Travesset, Joey Veltri}
\date{April 2022}
\begin{document}

\begin{abstract}
    In a recent work of Anderson and Hu, the authors constructed a measure that was $p$-adic and $q$-adic doubling, for any primes $p$ and $q$, yet not doubling.  This work relied heavily on a developed number theory framework.  Here we develop this framework further, which yields a measure that is $m$-adic and $n$-adic doubling for any coprime $m,n$, yet not doubling.  Additionally, we show several new applications to the intersection of weight and function classes.
\end{abstract}

\maketitle
\section{Introduction}
Measures and weight and function classes are frequent tools used in harmonic analysis and its applications (such as in \cite{DCU}, \cite{P}, \cite{P2}, to name a few).  The study of intersection properties of different classes of these objects is delicate, but results greatly increase our understanding of how these classes interact, uncovering their underlying structure (see \cite{LN}, \cite{LPW}, \cite{PW}).  For the case of measures, Boylan, Mills, and Ward recently constructed a measure that was both \emph{dyadic doubling} and \emph{triadic doubling} yet not doubling \cite{BMW}.  The doubling property and its $n$-adic analogue are important features of measures, and classically it has been known that dyadic doubling does not imply doubling.  However, extending this result has been difficult, and to the best of our knowledge, \cite{BMW} was the first major progress to extending this classical result to intersections.  Anderson and Hu were able to extend \cite{BMW} to any pairs of primes, by approaching both the underlying number theory and analysis from a different angle.  For example, \cite{BMW} relied on very specific facts about the primes $2$ and $3$ that do not generalize.  The difficulties overcome in \cite{AH} are outlined there, along with the difficulties that would need to be overcome to further extend those results from pairs of primes to coprime bases.  In particular, \cite{AH} says, ``While the construction of the measure and the analysis employed in Sections 4–7 could carry through in this [coprime] setting, we would still crucially rely on the underlying number theory connected to the geometry of this setting, where it appears that several new ideas would be needed."  In this paper, we supply the new ideas needed on the number theory side, which immediately gives the following extension of the main result of \cite{AH}.
\begin{thm}
\label{main result}
There exists an infinite family of measures that are both $m$-adic and $n$-adic doubling for any $(m,n)=1$ but not doubling.
\end{thm}

While the basic outline of the number theory is similar, relying on stability of certain orders in groups, the underlying structures the argument relies on are different.  For example, here we need to introduce a function $\psi(m,n)$ of two variables which governs the choice of certain progressions crucial to the construction of the measure $\mu$; this reduces to simply a power of a prime $p$ in \cite{AH}.  Additionally, several components of the proofs of this order stability and existence of progressions were much simpler in \cite{AH}.  Our new number theory approach is compared with \cite{AH} for easy reference, and we refer the reader to this source for more information on how this number theory directly improves the analysis, yielding Theorem \ref{main result}. 

We also discuss several applications, some of which are immediate given the new number theory and the framework of \cite{AH} and some of which are completely new.  For the immediate applications, we are able to show that the intersection of any $m$-adic \emph{reverse H\"older} weight class with any $n$-adic one, where $m$ and $n$ are coprime, is never the full reverse H\"older class (outside of the case of the $\infty$-reverse H\"older class, $RH_\infty$).  We have similar results for the Muckenhoupt $A_p$ weights.  For the extremal classes $RH_\infty$ and $A_1$ (see Section 3 for definitions and context), we further show that the weight used to prove the previous non-equality of the weight classes is not an $A_1$ nor an $RH_\infty$ weight, even $n$-adically.

Our most extensive application is the new addition of results for the $BMO$ class of functions of \emph{bounded mean oscillation}.  These are functions that do not deviate much from their averages, on average.  Formally, a function $f$ is in $BMO$ if and only if
\begin{equation}
    \|f\|_{BMO} := \sup_{I}\fint_I\left|f-\fint_If\right| < \infty ,
\end{equation}
where $I$ is any interval and $\fint_I$ gives the average (with respect to Lebesgue measure).  The class $BMO_n$ is defined similarly, except that the allowable intervals $I$ are restricted to be $n$-adic. Results about this class have inspired several folkloric questions about measures, answered in part by \cite{BMW} and \cite{AH}.  We found a reference in Krantz \cite{Krantz} to unpublished work of Peter Jones on intersections of $BMO_p$ classes, for $p$ prime; however, we have been unable to locate this work.  Krantz's work presents Jones's result for the specific case of the primes $2$ and $3$, specifically that
\begin{prop}
\label{main prop}
We have 
\[
BMO_2\cap BMO_3 \neq BMO.
\]
\end{prop}

To show this, Krantz explicitly constructs a sequence of functions $f_k$.  However, the details are sparse, so we decided to take a completely different approach involving constructing a function $f$ derived from the weight $w$ constructed in \cite{AH}.  Using this approach, we were able to extend Proposition \ref{main prop} to coprime bases.  Since this application is new, we first show the result for prime bases and then extend to the coprime case.  Our result is the following:
\begin{thm}
\label{BMO thm}
Choose any $(m,n)=1$.  We have 
\[
BMO_m\cap BMO_n \neq BMO.
\]
\end{thm}

Throughout this paper, we heavily refer to constructions and background in \cite{AH} and include comparisons and contrasts.  Section 2 contains the heavy number theory lifting, and may be of independent interest.  Section 3 contains all the applications, both the immediate (given the extensive framework in \cite{AH}) and the new.  Many interesting open questions remain for future investigation.

\subsection{Acknowledgements} The first author was supported by NSF DMS 1954407.  The authors would like to thank Nikos Villareal Styles for helpful discussions during this project.

\section{Coprime bases, number theory and main result}

Our first proposition is a substitute for Proposition 2.1 in \cite{AH}.  Here the role of the prime $p$ is played by $m$ and the prime $q$ is played by $n$, where $(m,n)=1$.  While this gives a stability result for certain orders, the structure that this stability takes is captured by a function $\psi(m,n)$, whose definition appears at the end of the proof.  The important thing is that this stability is independent of the parameter $t$ defined below (which was called $m$ in \cite{AH}, not to be confused with the different use of $m$ here).

Before stating the proposition, we first give a lemma that held trivially in \cite{AH} when $m$ was prime but requires further justification now that $m$ and $n$ are merely coprime.

\begin{lem}
\label{lem}
    If $(m, n) = 1$, then $m \mid \binom{m}{2}\left(\frac{n^{\phi(m)} - 1}{m}\right)^2$.
\begin{proof}
    Since we have $\binom{m}{2}\left(\frac{n^{\phi(m)} - 1}{m}\right)^2 = \frac{(m - 1)(n^{\phi(m)} - 1)^2}{2m}$, we would like to show that $2m^2 \mid (m - 1)(n^{\phi(m)} - 1)^2$. If $m$ is odd, then $m - 1$ is even, so the result follows since $m \mid n^{\phi(m)} - 1$. Suppose, then, that $m$ is even so that $n$ must be odd. We will prove that $n^{\phi(m)} \equiv 1 \MOD{2m}$, from which the result follows. \\
    \\
    Consider any odd $r$.
    Since $n^{\phi(r)} \equiv 1 \MOD r$, there is some integer $j$ such that $n^{\phi(r)} = jr + 1$. If $j$ were odd, then $n^{\phi(r)}$ would be even, a contradiction since $n$ is odd. Hence $j$ must be even, so $n^{\phi(r)} \equiv 1 \MOD{2r}$. \\
    \\
    Now suppose that $n^{\phi(2^\l r)} \equiv 1 \MOD{2^{\l + 1}r}$ for some nonnegative $\l$, which we have just shown holds for $\l = 0$. Then writing $n^{\phi(2^\l r)} = j2^{\l + 1}r + 1$ for some $j$,
    \begin{align*}
        n^{\phi(2^{\l + 1}r)} &= n^{2^\l\phi(r)} = n^{2\phi(2^\l r)} = \left(n^{\phi(2^\l r)}\right)^2 = (j2^{\l + 1}r + 1)^2 \\
        &= j^22^{2\l + 2}r^2 + j2^{\l + 2}r + 1 \equiv 1 \MOD{2^{\l + 2}r}.
    \end{align*}
    Thus, the claim holds for all nonnegative $\l$, which means that $n^{\phi(m)} \equiv 1 \MOD{2m}$ since $m = 2^\l r$ for some $\l \in \N$ and odd $r$.
\end{proof}
\end{lem}

Now we proceed to generalize the propositions from \cite{AH}.

\begin{prop}
\label{prop1}
Let $m, n \geq 2$ be coprime with $m > n$. Further, let $O_t(m, n)$ denote the order of $n^{\phi(m)}$ in $(\Z / m^t\Z)^*$. Then there is some $\psi(m, n)$ such that for all sufficiently large $t$, we have $$\psi(m, n) = \frac{m^t}{O_t(m, n)}.$$
\end{prop}
\begin{proof}
First let us define $t(m, n)$ to be the smallest positive integer such that $n^{\phi(m)} \not\equiv 1 \MOD{m^{t(m, n) + 1}}$. (There must be some such integer since $m^{t + 1}$ will exceed $n^{\phi(m)} > 1$ for large enough $t$.) \\
\\
Lemma: For all $\l \geq 0$, we have
$$\left(n^{\phi(m)}\right)^{m^\l} \not\equiv 1 \MOD {m^{t(m, n) + \l + 1}}.$$
Clearly this holds for $\l = 0$ by the definition of $t(m, n)$. Now suppose it holds for some $\l \geq 0$. Note that $n^{\phi(m)} \equiv 1 \MOD{m^{t(m, n)}}$ by Euler's Theorem when $t(m, n) = 1$, and this also holds by the definition of $t(m, n)$ when $t(m, n) > 1$. This gives us
\begin{align*}
    \left(n^{\phi(m)}\right)^{m^\l} \equiv 1 \MOD{m^{t(m, n) + \l}}. \tag{$*$}
\end{align*}
So for some $s$ we can write
\begin{align*}
    \left(n^{\phi(m)}\right)^{m^\l} = m^{t(m, n) + \l} \cdot s + 1.
\end{align*}
Note that we must have $m \nmid s$ because otherwise, this contradicts the inductive hypothesis. Therefore,
\begin{align*}
    \left(n^{\phi(m)}\right)^{m^{\l + 1}} &= \left(\left(n^{\phi(m)}\right)^{m^\l}\right)^m \\
    &= \left (m^{t(m,n) + \l} \cdot s + 1\right)^m \\
    &= \sum_{i = 0}^m \binom{m}{i} \left(m^{t(m, n) + \l} \cdot s\right)^i \\
    &= 1 + m\left(m^{t(m, n) + \l} \cdot s\right) + \sum_{i = 2}^m \binom{m}{i} \left(m^{t(m, n) + \l} \cdot s\right)^i \\
    &= 1 + m^{t(m, n) + \l + 1} \cdot s + \sum_{i = 2}^m \binom{m}{i} \left(m^{i(t(m, n) + \l)} \cdot s^i\right).
\end{align*}
Now note that for $i \geq 3$,
\begin{align*}
    i(t(m, n) + \l) &\geq t(m, n) + \l + 2(t(m, n) + \l) \\
    &\geq t(m, n) + \l + 2(1 + 0) \\
    &= t(m, n) + \l + 2,
\end{align*}
so $m^{t(m, n) + \l + 2}$ divides every term in the sum with $i \geq 3$. $m^{t(m, n) + \l + 2}$ likewise divides the term with $i = 2$ if $t(m, n) > 1$ or $\l > 0$. Otherwise, if $t(m, n) = 1$ and $\l = 0$, then we have $s = \frac{n^{\phi(m)} - 1}{m}$, in which case Lemma \ref{lem} ensures that $m^{t(m, n) + \l + 2}$ divides the term with $i = 2$ regardless since
\begin{align*}
    m^{t(m, n) + \l + 2} = m^2 \cdot m \mid m^2 \cdot \binom{m}{2}\left(\frac{n^{\phi(m)} - 1}{m}\right)^2 = \binom{m}{2}\left(m^{2(t(m, n) + \l)} \cdot s^2\right).
\end{align*}
Therefore, every term in the sum is divisible by $m^{t(m, n) + \l + 2}$. Furthermore, since $m \nmid s$, we also know $m^{t(m, n) + \l + 2} \nmid m^{t(m, n) + \l + 1} \cdot s$. This gives us
\begin{align*}
    (n^{\phi(m)})^{m^{\l + 1}} &= 1 + m^{t(m, n) + \l + 1} \cdot s + \sum_{i = 2}^m \binom{m}{i} \left(m^{i(t(m, n) + \l)} \cdot s^i\right) \\
    &\equiv 1 + m^{t(m, n) + \l + 1} \cdot s \\
    &\not\equiv 1 \MOD {m^{t(m, n) + \l + 2}}.
\end{align*}
So since our claim holds for $\l + 1$ if it holds for $\l$, this proves our lemma. Now by substituting $\l + 1$ for $\l$ in $(*)$, $O_{t(m, n) + \l + 1}(m, n) \mid m^{\l + 1}$, and by the lemma, $O_{t(m, n) + \l + 1}(m, n) \nmid m^\l$. Hence for any $\l \geq 0$, $O_{t(m, n) + \l + 1}(m, n) = \gamma_\l m^\l$, where $\gamma_\l \mid m$ and $\gamma_\l > 1$. Moreover, since $\left(n^{\phi(m)}\right)^{\gamma_\l m^\l} \equiv 1 \MOD{m^{t(m, n) + \l + 1}}$, we have $\left(n^{\phi(m)}\right)^{\gamma_\l m^{\l + 1}} \equiv 1 \MOD{m^{t(m, n) + \l + 2}}$, so
$$\gamma_{\l + 1}m^{\l + 1} = O_{t(m, n) + \l + 2}(m, n) \mid \gamma_\l m^{\l + 1}.$$
This means that $\gamma_{\l + 1} \mid \gamma_\l$ and hence the sequence $(\gamma_\l)$ is monotone decreasing. Also, it only takes on a finite number of values since there are only a finite number of positive divisors of $m$, so there must come a point $L$ such that for all $\l \geq L$, $\gamma_\l = \gamma_L$. Let us denote $\gamma_L$ simply by $\gamma$. Now if $t = t(m, n) + \l + 1$ so that $\l = t - t(m, n) - 1$, as long as $t \geq t(m, n) + L + 1$ so that $\l \geq L$, this gives us
$$O_t(m, n) = \gamma m^{t - t(m, n) - 1}.$$
Therefore,
$$\psi(m, n) = \frac{m^t}{O_t(m, n)} = \frac{m^t}{\gamma m^{t - t(m, n) - 1}} = \frac{m^{t(m, n) + 1}}{\gamma}$$
does not depend on $t$.
\end{proof}
\begin{rem}
Note that in \cite{AH}, because $m$ was prime, there were no factors of $m$ greater than 1 besides $m$ itself, which guaranteed that $\psi(m, n) = m^{t(m, n)}$ was a power of $m$ and that $L = 0$. Also, $N_0$ as defined in \cite{AH} was determined to be in fact 1, simplifying the proof considerably.
\end{rem}

The next proposition does not have an analogue in \cite{AH} since, as remarked after the proof, the analogous result followed immediately.
\begin{prop}
\label{prop2}
Let $m, n \geq 2$ be coprime with $m > n$. Let $\psi(m,n)$ be defined as in Proposition \ref{prop1}. Then
$$n^{\phi(m)} \equiv 1 \MOD{\psi(m, n)}.$$
\end{prop}
\begin{proof}
Recall that $\psi(m, n) = m^{t(m, n) + 1} / \gamma$ for some $\gamma \mid m$, where $\gamma > 1$. From Proposition \ref{prop1}, $n^{\phi(m)} \equiv 1 \MOD{m^{t(m, n)}}$, so we can write $n^{\phi(m)} = m^{t(m, n)} \cdot s + 1$. Then for any $t \geq t(m, n) + L + 1$, since $O_t(m, n) = \gamma m^{t - t(m, n) - 1}$,
\begin{align*}
    0 &\equiv \left(n^{\phi(m)}\right)^{\gamma m^{t - t(m, n) - 1}} - 1 \\
    &= \left(m^{t(m, n)} \cdot s + 1\right)^{\gamma m^{t - t(m, n) - 1}} - 1 \\
    &= \sum_{i = 0}^{\gamma m^{t - t(m, n) - 1}}\binom{\gamma m^{t - t(m, n) - 1}}{i}\left(m^{t(m, n)} \cdot s\right)^i - 1 \\
    &= \sum_{i = 1}^{\gamma m^{t - t(m, n) - 1}}\binom{\gamma m^{t - t(m, n) - 1}}{i}\left(m^{t(m, n)} \cdot s\right)^i \MOD{m^t}.
\end{align*}
Now if $2 \leq i \leq \gamma m^{t - t(m, n) - 1} - 1$, since $(i - 1)t(m, n) \geq i - 1 \geq 1$, we have
\begin{align*}
    m^t &\mid \gamma m^t \mid \gamma m^{t + (i - 1)t(m, n) - 1} = \gamma m^{t - t(m, n) - 1}m^{it(m, n)} \mid \binom{\gamma m^{t - t(m, n) - 1}}{i}m^{it(m, n)}.
\end{align*}
As long as $t$ is sufficiently large, we also have
$$m^t \mid \left(m^{t(m, n)}\right)^{\gamma m^{t - t(m, n) - 1}}.$$
Therefore, we can drop all terms with $i \geq 2$ from the sum to get
\begin{align*}
    0 &\equiv \gamma m^{t - t(m, n) - 1} \cdot m^{t(m, n)} \cdot s = \gamma m^{t - 1} \cdot s \MOD{m^t}.
\end{align*}
So we must have $m \mid \gamma s \Rightarrow \frac{m}{\gamma} \mid s$ and hence since $\psi(m, n) = m^{t(m, n) + 1} / \gamma$, we see that $n^{\phi(m)} \equiv 1 \MOD{\psi(m, n)}$.
\end{proof}
\begin{rem}
Note that in \cite{AH}, this result was a trivial consequence of the definition of $t(m, n)$, for as remarked above, $\psi(m, n) = m^{t(m, n)}$ when $m$ is prime, giving $n^{\phi(m)} \equiv 1 \MOD{\psi(m, n)}$ as in Proposition \ref{prop1}.
\end{rem}

The final proposition of this section is an analogue of Proposition 2.3 in \cite{AH}, a thorough commentary of the derivation and significance of which appears in a remark therein.  The main idea is that this proposition guarantees the existence of certain arithmetic progressions, as stated, which immediately leads to a geometric proximity of corresponding points on $m$-adic and $n$-adic intervals.  Precisely, this proposition is the main input to guarantee that certain distinguished points $\Upsilon$ and $\textZeta$ are ``within $\epsilon$", as shown in Proposition 3.5 of \cite{AH}.  The proximity of these points governs how $m$-adic and $n$-adic intervals can intersect, which underlies the analysis done in Sections 4–7 of \cite{AH}, particularly in the extensive work to show that the constructed measure is $m$-adic doubling (which was $p$-adic doubling in the language of \cite{AH}).  The new features here include the different structure of the progressions, and a simplified proof that avoids some of the group theory as compared to \cite{AH}.
\begin{prop}
\label{prop3}
Let $m, n \geq 2$ be coprime with $m > n$. Then for any sufficiently large $t_1 \in \N$ and any $k \in G_{t_1}(m, n)$, where
\begin{align*}
    G_{t_1}(m, n) &= \{1 + i\psi(m, n): i \in \{0, ..., m^{t_1\phi(n)} / \psi(m, n) - 1\}\} \\
    &= \{a \in [1, m^{t_1\phi(n)}]: a \equiv 1 \MOD{\psi(m, n)}\}\text{,}
\end{align*}
there exist infinitely many $(t_2, j) \in \N^2$, where
\begin{align*}
    j &\in \{in - 1: i \in \{1, ..., n^{t_2\phi(m) - 1}\}\} \\
    &= \{b \in [1, n^{t_2\phi(m)}]: b \equiv -1 \MOD n\}\text{,}
\end{align*}
such that
\begin{align*}
    \frac{k}{m^{t_1\phi(n)}} - \frac{j}{n^{t_2\phi(m)}} &= \frac{1}{m^{t_1\phi(n)}n^{t_2\phi(m)}}.
\end{align*}
\end{prop}
\begin{proof}
By Proposition \ref{prop2}, $n^{\phi(m)} \equiv 1 \MOD{\psi(m,n)}$, so $\cs{n^{\phi(m)}} \leq (\Z / m^{t_1\phi(n)}\Z)^*$ is a subset of $G_{t_1}(m, n)$. Moreover, by Proposition \ref{prop1}, for any $t \geq t(m, n) + L + 1$, the order of $n^{\phi(m)}$ modulo $m^t$ is $m^t / \psi(m, n)$. So if $t_1 \geq \frac{t(m, n) + L + 1}{\phi(n)}$, we will have $|n^{\phi(m)}| = |G_{t_1}(m, n)|$. Hence $G_{t_1}(m, n) = \cs{n^{\phi(m)}}$, so for any $k \in G_{t_1}(m, n)$, there is some nonnegative $t' < |G_{t_1}(m, n)|$ such that
\begin{align*}
    \left(n^{\phi(m)}\right)^{t'} \equiv k \MOD{m^{t_1\phi(n)}}.
\end{align*}
Therefore,
\begin{align*}
    kn^{t_2\phi(m)} &\equiv \left(n^{\phi(m)}\right)^{t'}\left(n^{\phi(m)}\right)^{t_2} = \left(n^{\phi(m)}\right)^{t' + t_2} \MOD{m^{t_1\phi(n)}}.
\end{align*}
Now for any $t_2 \in \{i|G_{t_1}(m, n)| - t': i \in \N\}$, we have $kn^{t_2\phi(m)} \equiv 1 \MOD{m^{t_1\phi(n)}}$. Hence for any such $t_2$, there is some $j$ such that
$$kn^{t_2\phi(m)} - jm^{t_1\phi(n)} = 1 \ \Rightarrow \ \frac{k}{m^{t_1\phi(n)}} - \frac{j}{n^{t_2\phi(m)}} = \frac{1}{m^{t_1\phi(n)}n^{t_2\phi(m)}}.$$
Clearly, $j \geq 1$ because otherwise, if $j \leq 0$, we would have
\begin{align*}
    kn^{t_2\phi(m)} - jm^{t_1\phi(n)} &\geq kn^{t_2\phi(m)} > 1. 
\end{align*}
Also, $j \leq n^{t_2\phi(m)}$ because otherwise, if $j > n^{t_2\phi(m)}$, since $k \leq m^{t_1\phi(n)}$, we would have
\begin{align*}
    1 &= kn^{t_2\phi(m)} - jm^{t_1\phi(n)} < kn^{t_2\phi(m)} - n^{t_2\phi(m)}m^{t_1\phi(n)} = n^{t_2\phi(m)}\left(k - m^{t_1\phi(n)}\right) \leq 0.
\end{align*}
Moreover, since $m^{t_1}$ and $n$ are coprime, $\left(m^{t_1}\right)^{\phi(n)} \equiv 1 \MOD n$, which means
\begin{align*}
    j &\equiv j\left(m^{t_1}\right)^{\phi(n)} \equiv jm^{t_1\phi(n)} - kn^{t_2\phi(m)} = -1 \MOD n.
\end{align*}
So $j \in \{b \in [1, n^{t_2\phi(m)}]: b \equiv -1 \MOD n\}$, which concludes the proof.
\end{proof}

As stated in the introduction, using the framework of \cite{AH}, these propositions immediately imply Theorem \ref{main result}.

\section{Applications to structure of weight and function classes}
We now show a wide variety of applications to weight and function classes.  The first few are immediate consequences of the number theory results in the previous section and the proofs of the prime case in \cite{AH}.  The last several applications are new and are discussed in more detail.

Firstly, we have applications for reverse H\"older and $A_p$ weights (which we will call $A_r$ weights, following the notation of \cite{AH}).
Define the \emph{reverse H\"older and $n$-adic reverse H\"older classes} as follows:
\begin{defn}
Let $r>1$. We say that $w\in RH_r$ if 
\begin{equation}
\label{20200829eq01}
\left(\fint_I w^r \right)^{\frac{1}{r}} \leq C \fint_I w 
\end{equation}
for all intervals $I$, where $C$ is an absolute constant. We say $w \in RH_1$ if $w \in RH_r$ for some $r>1$, that is
$$
RH_1:=\bigcup_{r>1} RH_r. 
$$
\end{defn}
\begin{defn}
Let $r>1$. We say that $w\in RH_r^n$ if
\begin{equation} \label{20200829eq02}
\left(\fint_Q w^r \right)^{\frac{1}{r}} \leq C \fint_Q w
\end{equation} 
for all $n$-adic intervals $Q$, where $C$ is an absolute constant and $w$ is $n$-adic doubling. Moreover, we say $w \in RH^n_1$ if $w \in RH_r^n$ for some $r>1$, that is
$$
RH^n_1:=\bigcup_{r>1} RH_r^n. 
$$
\end{defn}

\begin{cor} \label{20200901cor01}
For any $r>1$, \[RH_r^m \cap RH_r^n \neq RH_r.\]
In particular, 
\[RH_1^m \cap RH_1^n \neq RH_1.\]
\end{cor}

\begin{defn}
Let $1<r<\infty$,  we say a weight $w\in A_r$ if 
\[
\sup_I \left(\fint_I w(x)dx\right)\left(\fint_I w(x)^{\frac{-1}{r-1}}dx\right)^{r-1} < \infty,
\]
where the supremum is taken over all intervals $I$. Moreover, we say $w \in A_\infty$ if $w \in A_r$ for some $r>1$, that is,
$$
A_\infty:=\bigcup_{r>1} A_r. 
$$
\end{defn}
 We define the \emph{$m$-adic} $A_r^m$ and \emph{$n$-adic} $A^n_\infty$ similarly by only allowing averages along $m$-adic or $n$-adic intervals.  Note that the $A_r$ condition implies doubling.
\begin{cor} \label{Ap cor}
For any $r>1$ and $(m,n)=1$, we have 
$$
A_r^m \cap A_r^n \neq A_r.
$$
In particular,
$$
A_\infty^m \cap A_\infty^n \neq A_\infty.
$$
 \end{cor}

We also have some additional complements to the listed results above for the ``extremal" cases of $RH_\infty$ and $A_1$ weights. 
Recall that we say $w\in RH_\infty$ if and only if

$$\sup_I\left(\esssup_{x \in I}w_\mu\right)\left(\fint_I w_\mu\right)^{-1}  < \infty$$
and $w\in A_1$ if and only if

$$\sup_I\left(\fint_I w_\mu\right) \left( \frac{1}{\essinf_{x 
\in I}w_\mu(x)}\right) < \infty.$$
We define $RH_\infty^n$ and $A_1^n$ analogously.
Here let $w$ be the weight from the measure $\mu$ constructed in Section 4 of  \cite{AH}, which is completely defined on $q$-adic intervals by $\int_Iw(x)dx = \mu(I)$, and recall the definition of the \emph{modules} $I^\alpha$.
\begin{thm}
The weight $w_\mu \notin RH_\infty^q$.
\begin{proof}
Let $Q$ be some $q$-adic interval that intersects one or more of the modules $I_\ell^{\alpha_\ell}$. Set $\alpha$ to be the largest $\alpha_\ell$ of all the intersected modules. Then, $\fint_Q w = 1$ and $\esssup_{x \in Q} w_\mu(x) = \max_{x \in Q} w_\mu(x) = b^{\alpha + 1}$. Thus
$$\sup_Q\left(\esssup_{x \in Q}w_\mu\right)\left(\fint_Q w_\mu\right)^{-1} \geq b^{\alpha + 1}.$$
Since $\alpha$ can be made arbitrarily large, this means that 
$$\sup_Q\left(\esssup_{x \in Q}w_\mu\right)\left(\fint_Q w_\mu\right)^{-1}  \nless \infty,$$
so $w_\mu \notin RH_\infty^q$.
\end{proof}
\end{thm}
Note that since the above does not depend on the fact that $q$ is prime, we have immediately that
\begin{cor}
The weight $w_\mu \notin RH_\infty^n$.
\end{cor}

\begin{thm}
The weight $w_\mu \notin A_1^q$.
\begin{proof}
Let $Q$ be some $q$-adic interval that intersects one or more of the modules $I_\ell^{\alpha_\ell}$. Set $\alpha$ to be the largest $\alpha_\ell$ of all the intersected modules. Then, $\fint_Q w = 1$ and $\essinf_{x \in Q}w_\mu = \min_{x \in Q}w_\mu(x) = a^{\alpha +1}$. Thus

$$\sup_Q\left(\fint_Q w_\mu\right) \left( \frac{1}{\essinf_{x 
\in Q}w_\mu(x)}\right) \geq \frac{1}{a^{\alpha + 1}}.$$
\noindent Since $\alpha$ can be made arbitrarily large, this means that

$$\sup_Q\left(\fint_Q w_\mu\right) \left( \frac{1}{\essinf_{x 
\in Q}w_\mu(x)}\right) \nless \infty,$$

\noindent so $w_\mu \notin A_1^q$.
\end{proof}
\end{thm}
Similarly, since nothing depends on the fact that $q$ is prime we have
\begin{cor}
The weight $w_\mu \notin A_1^n$.
\end{cor}

We will now prove a new application about the BMO function class.  As mentioned in the introduction, we give a number-theoretic proof that the intersection of two of these ``prime" BMO classes is never BMO, and extend this immediately, via our number theory, to the coprime case.  We heavily refer to our notation and definitions in \cite{AH}.  There are other possible applications and extensions, such as to the VMO class of functions or to the Hardy space, but we choose not to pursue these here. 
\begin{thm}
For any primes $p$ and $q$, $BMO_p \cap BMO_q \neq BMO$.
\begin{proof}
To prove this, we show that there exists a function in $BMO_p$ and $BMO_q$ but not $BMO$. Consider $f(x) = \log w(x)$, where $w=w_\mu$. Recall from \cite{AH} the definition of the modules $I^\alpha$, the values $a$ and $b$, and the intervals $H^\alpha$ and $G^\alpha$.  We know that $w \in A_\infty^p \cap A_\infty^q$, so by Corollary 2.19 from Chapter IV of \cite{GCRF}, which remains valid when restricted to the $p$-adic and $q$-adic cases, $f \in BMO_p \cap BMO_q$. Use the notation $\log\left(\frac{x}{q}\right)^\alpha$ for $\log\left(\left(\frac{x}{q}\right)^\alpha\right)$.  Now, note that
\begin{align*}
    \frac{1}{|H^{(\alpha)} \sqcup G^{(\alpha)}|}\int_{H^{(\alpha)} \sqcup G^{(\alpha)}} f &= \frac{1}{|H^{(\alpha)} \sqcup G^{(\alpha)}|}\left(\int_{H^{(\alpha)}} \log \left(\frac{a}{q}\right)^\alpha dx + \int_{G^{(\alpha)}} \log \left(\frac{b}{q}\right)^\alpha dx\right) \\
    &= \frac{1}{|H^{(\alpha)} \sqcup G^{(\alpha)}|}\left(|H^{(\alpha)}| \log \left(\frac{a}{q}\right)^\alpha + |G^{(\alpha)}| \log \left(\frac{b}{q}\right)^\alpha\right) \\
    &= \frac{1}{2}\log \left(\frac{a}{q}\right)^\alpha + \frac{1}{2}\log \left(\frac{b}{q}\right)^\alpha \\
    &= \frac{1}{2}\log \left(\frac{ab}{q^2}\right)^\alpha.
\end{align*}
Therefore,
\begin{align*}
    ||f||_{BMO} &= \sup_I \frac{1}{|I|}\int_I \left|f - \frac{1}{|I|}\int_I f\right| \\
    &\geq \frac{1}{|H^{(\alpha)} \sqcup G^{(\alpha)}|}\int_{H^{(\alpha)} \sqcup G^{(\alpha)}} \left|f - \frac{1}{|H^{(\alpha)} \sqcup G^{(\alpha)}|}\int_{H^{(\alpha)} \sqcup G^{(\alpha)}} f\right| \\
    &\geq \frac{1}{|H^{(\alpha)} \sqcup G^{(\alpha)}|}\int_{G^{(\alpha)}} \left|f - \frac{1}{|H^{(\alpha)} \sqcup G^{(\alpha)}|}\int_{H^{(\alpha)} \sqcup G^{(\alpha)}} f\right| \\
    &= \frac{1}{|H^{(\alpha)} \sqcup G^{(\alpha)}|}\int_{G^{(\alpha)}} \left|\log\left(\frac{b}{q}\right)^\alpha - \frac{1}{2}\log \left(\frac{ab}{q^2}\right)^\alpha\right| \ dx \\
    &= \frac{|H^{(\alpha)}|}{|H^{(\alpha)} \sqcup G^{(\alpha)}|} \left|\log\left(\frac{b}{q}\right)^\alpha - \frac{1}{2}\log \left(\frac{ab}{q^2}\right)^\alpha\right| \\
    &= \frac{1}{2}\left|\log\frac{(b/q)^\alpha}{(ab/q^2)^{\alpha / 2}}\right| \\
    &= \frac{1}{2}\left|\log\left(\frac{b}{a}\right)^{\alpha/2}\right| \\
    &= \frac{\alpha}{4}\log\frac{b}{a}.
\end{align*}
So since $\alpha$ can be arbitrarily large and $b>1>a>0$, $||f||_{BMO} = \infty$, which implies that $f \not\in BMO$, as was to be shown.
\end{proof}
\end{thm}

We now restate Theorem \ref{BMO thm} as a corollary of this result.
\begin{cor}
For any coprime integers $m, n \geq 2$, $BMO_m \cap BMO_n \neq BMO$.
\begin{proof}
This results from the number-theoretic results proved previously, applied to the construction in \cite{AH}, and run through the proof of the previous theorem.  Indeed, define the function $f$ as in the previous theorem.  We still have that $f \in BMO_m \cap BMO_n$; we will show that $f \notin BMO$ by the same approach.  One can modify the construction in \cite{AH}, given the number theoretic results proved earlier, which simply substitutes $q$ for $n$.  Hence the proof that $f\notin BMO$ as carried out above translates exactly to this case by replacing the instances of $q$ by $n$, which gives the result.
\end{proof}
\end{cor}

\end{document}